\documentclass[a4paper]{article}

\usepackage[english]{babel}
\usepackage[utf8x]{inputenc}
\usepackage{amsmath}

\usepackage[a4paper,top=3cm,bottom=2cm,left=3cm,right=3cm,marginparwidth=1.75cm]{geometry}

\usepackage{amsthm}
\usepackage{amsfonts}
\usepackage{graphicx}

\usepackage[colorlinks=true, allcolors=blue]{hyperref}
\usepackage{url}
\usepackage{float}
\restylefloat{table}

\newtheorem{question}{Question}

\newtheorem{lemma}{Lemma}

\newtheorem{theorem}{Theorem}

\newtheorem{example}{Example}

\title{A note on the use of Rédei polynomials for solving the polynomial Pell equation and its generalization to higher degrees}

\author{Nadir Murru \\
Department of Mathematics G. Peano, University of Torino\\
Via Carlo Alberto 10, 10123, Torino, ITALY\\
nadir.murru@unito.it}
\date{}

\begin{document}
\maketitle

\begin{abstract}
The polynomial Pell equation is
\[P^2 - D Q^2 = 1\]
where $D$ is a given integer polynomial and the solutions $P, Q$ must be integer polynomials. A classical paper of Nathanson \cite{Nat} solved it when $D(x) = x^2 + d$. We show that the Rédei polynomials can be used in a very simple and direct way for providing these solutions. Moreover, this approach allows to find all the integer polynomial solutions when $D(x) = f^2(x) + d$, for any $f \in \mathbb Z[X]$ and $d \in \mathbb Z$, generalizing the result of Nathanson.
We are also able to find solutions of some generalized polynomial Pell equations introducing an extension of Rédei polynomials to higher degrees.
\end{abstract}


\section{Introduction}

The Pell equation 
\[x^2 - d y^2 = 1\]
is one of the most famous Diophantine equations. It has infinite integer solutions when $d$ is not a square and they can be determined using the continued fraction's expansion of $\sqrt{d}$. It is very interesting to study the polynomial Pell equation, i.e., find polynomials $P$ and $Q$ in $\mathbb Z[X]$ satisfying 
\begin{equation} \label{eq:pell} P^2 - D Q^2 = 1 \end{equation}
where $D \in \mathbb Z[X]$ is a fixed polynomial. Nathanson \cite{Nat} solved the polynomial Pell equation when $D = x^2 + d$. In particular, he proved that it has non-trivial solutions if and only if $d = \pm 1, \pm 2$ and provided explicitly the polynomial solutions. Pastor \cite{Pas} showed that the solutions of Nathanson can be also expressed in terms of Chebyshev polynomials. Webb and Yokota \cite{WY1} found a necessary and sufficient condition for which the polynomial Pell equation has non-trivial solutions when $D = A^2 + 2 C$ monic polynomial, $A/C \in \mathbb Z[X]$ and $\deg C < 2$. In \cite{WY2}, such result is generalized when $pA/C \in \mathbb Z[X]$, for some prime $p$, without any condition on the degree of $C$. In this case, the authors also determined the solutions. Then, Yokota \cite{Yok} found a necessary and sufficient condition for the solution of the polynomial Pell equation when $A/C \in \mathbb Q[X]$. Hazama \cite{Haz} studied the polynomial Pell equation using the twist of a conic by another conic. Further studies can be found in \cite{Mol} and \cite{Ram}. Mc Laughlin \cite{Mac} focused on the relations between polynomial solutions of the Pell equation and fundamental units of real quadratic fields. Some authors studied polynomial solutions of 
\eqref{eq:pell} in $\mathbb C[X]$, see \cite{Dub}, \cite{Gri}, \cite{Zap}.
Despite many studies on the polynomial Pell equation, there are no works that highlight its connection with the Rédei polynomials, which are also classical tools in number theory, see \cite{Red}, \cite{Lidl}.
In this paper, we show this connection and how using Rédei polynomials for solving \eqref{eq:pell} for polynomials $D$ more general than the polynomials studied in previous works. 

The paper is structured as follows.
We introduce the Rédei polynomials in section \ref{sec:pell} where we show that solutions of \eqref{eq:pell} are the Rédei polynomials for some kind of polynomials $D$. The solutions of Nathanson arise as particular cases. The used method, even if simple and straightforward, is very general and allows to approach also polynomial Pell equations with higher degrees by means of a generalization of the Rédei polynomials, as we will see in section \ref{sec:higher}. 

\section{Solution of the polynomial Pell equation via Rédei polynomials} \label{sec:pell}

Rédei polynomials were introduced in \cite{Red} and they are a particular case of Dickson polynomials \cite{Lidl}. They arise from the development of
\[ (\sqrt{\alpha} + z)^n = D_n(\alpha, z) \sqrt{\alpha} + N_n(\alpha, z), \quad n = 0, 1, \ldots \]
where $z, \alpha \in\mathbb Z$, $\alpha$ not square. They can be written in the following closed form:
\[ N_n(\alpha, z) = \sum_{k=0}^{[n/2]}\binom{n}{2k}\alpha^k z^{n-2k}, \quad D_n(\alpha, z) = \sum_{k=0}^{[n/2]}\binom{n}{2k+1}\alpha^k z^{n-2k-1}.\]
It will be very useful the following matricial identity that can be easily proved by induction:
\begin{equation} \label{eq:redei-matrix} \begin{pmatrix} z & \alpha \cr 1 & z  \end{pmatrix}^n = \begin{pmatrix} N_n(\alpha, z) & \alpha D_n(\alpha, z) \cr D_n(\alpha, z) & N_n(\alpha, z) \end{pmatrix}, \end{equation}
from which follows
\begin{equation} \label{eq:matrices} \begin{pmatrix} z & \alpha \cr 1 & z  \end{pmatrix} \begin{pmatrix} N_{n-1}(\alpha, z) \cr D_{n-1}(\alpha, z) \end{pmatrix} = \begin{pmatrix} N_n(\alpha, z) \cr D_n(\alpha, z) \end{pmatrix}, \quad  \begin{pmatrix} N_{n-1}(\alpha, z) \cr D_{n-1}(\alpha, z) \end{pmatrix} = \begin{pmatrix} \cfrac{z}{z^2 - \alpha} & -\cfrac{\alpha}{z^2 - \alpha} \cr -\cfrac{1}{z^2 - \alpha} & \cfrac{z}{z^2 - \alpha}  \end{pmatrix} \begin{pmatrix} N_n(\alpha, z) \cr D_n(\alpha, z) \end{pmatrix}\end{equation}
From \eqref{eq:redei-matrix}, it also follows that the Rèdei polynomials are linear recurrent sequences with characteristic polynomial $t^2 - 2 z t + z^2 - \alpha$, i.e.,
\begin{eqnarray}\label{eq:redei-ric}\begin{cases}   
N_n(\alpha, z) = 2 z N_{n-1}(\alpha, z) - (z^2 - \alpha) N_{n-2}(\alpha, z), \quad n\geq2 \cr
N_0(\alpha, z) = 1, N_1(\alpha, z) = z 
\end{cases}, \\ \label{eq:redei-ric2}
\begin{cases}   
D_n(\alpha, z) = 2 z D_{n-1}(\alpha, z) - (z^2 - \alpha) D_{n-2}(\alpha, z), \quad n\geq2 \cr
D_0(\alpha, z) = 0, D_1(\alpha, z) = 1 
\end{cases}\end{eqnarray}
Finally, we can observe that $N_n(\alpha, z)$ and $D_n(\alpha, z)$ are polynomials of degree $n$ and $n-1$ in $z$, respectively. The Rédei rational functions $\cfrac{N_n(\alpha, z)}{D_n(\alpha, z)}$ have been studied and applied in several fields. For instance, they have been exploited to create public key cryptographic systems \cite{Bel}, \cite{Nob} and to generate pseudorandom sequences \cite{Top}. Moreover, in \cite{Barb}, the authors used the Rédei rational functions for generating solutions of the classical Pell equation. For further properties of the Rédei polynomials, see \cite{Lidl}.

In the classical paper of Nathanson \cite{Nat}, he proved that the polynomial Pell equation
\[ P^2 - (x^2 + d)Q^2 = 1, \]
has non-trivial solutions $P(x), Q(x) \in \mathbb Z[X]$ if and only if $d = \pm 1, \pm 2$ and he gave these polynomials explicitly:
\[ \begin{cases} 
A_n = \left( \frac{2}{d} x^2 + 1 \right)A_{n-1} + \frac{2}{d}x(x^2+d)B_{n-1}, \quad A_0 = 1 \cr
B_n = \frac{2}{d} x A_{n-1} + \left( \frac{2}{d}x^2 + 1 \right)B_{n-1}, \quad B_0 = 0
 \end{cases} \]
for any $n \geq 1$, when $d = 1, \pm 2$ and
\[ \begin{cases} 
A'_n = xA'_{n-1} + (x^2 - 1)B'_{n-1}, \quad A'_0 = 1 \cr
B'_n = A'_{n-1} + xB'_{n-1}, \quad B'_0 = 0
 \end{cases} \]
for any $n \geq 1$, when $d = -1$. Using the Rédei polynomials, we can solve a more general case of polynomial Pell equations, where we retrieve the results of Nathanson as a particular case. In the next theorem we see that the Rédei polynomials are polynomial solutions of
\begin{equation} \label{eq:pell-gen} P^2 - (f^2(x) + d)Q^2 = 1,\end{equation}
where $f(x)$ is any integer polynomial and $d \in \mathbb Z$.

\begin{theorem} \label{thm:main}
Given $d \in \mathbb Z$ and $f \in \mathbb Z[X]$, consider the Rédei polynomials $N_n(x) = N_n(f^2(x) + d, f(x))$ and $D_n(x) = D_n(f^2(x) + d, f(x))$, then 
\[ \left( \frac{N_n(x)}{(-d)^{n/2}} \right)^2 - (f^2(x) + d) \left( \frac{D_n(x)}{(-d)^{n/2}} \right)^2 = 1. \]
Moreover, the solutions are integer polynomials if and only if $d = 1, \pm 2$ and $n$ even, $d = -1$ for any $n$.
\end{theorem}
\begin{proof}
From \eqref{eq:redei-matrix} we have
\[ N_n^2(\alpha, z) - \alpha D_n^2(\alpha, z) = (z^2 - \alpha)^n, \]
for any $n \geq 1$. Considering $z = f(x)$ and $\alpha = f^2(x) + d$ we obtain
\[N_n^2(x) - (f^2(x) + d)D_n^2(x) = (-d)^n\]
i.e.,
\[\left(\frac{N_n(x)}{(-d)^{n/2}} \right)^2 - (f^2(x) + d) \left( \frac{D_n(x)}{(-d)^{n/2}} \right)^2 = 1.\] 
Clearly, when $d = -1$, we have obtained solutions of the polynomial Pell equation which are integer polynomials. When $d = 1$, $\cfrac{N_n(x)}{(-d)^{n/2}}$ and $\cfrac{D_n(x)}{(-d)^{n/2}}$ are integer polynomials if and only if $n$ is even. We complete the proof showing that $d^{[n/2]} \mid N_n(x)$ if and only if $d = \pm 2$, for any $n \geq 1$. For $n = 1$ and $n = 2$, we have
\[N_1(x) = f(x), \quad N_2(x) = 2 f^2(x) + d\]
and the statement is true. Let us observe that if $d \not= \pm 2$, then $d \not\mid N_2(x)$. Now, we proceed by induction. From \eqref{eq:redei-ric}, we have
\[N_n(x) = 2 f(x) N_{n-1}(x) + d N_{n-2}(x) = 2 f(x) d^{[(n-1)/2]}\cdot a + d \cdot d^{(n-2)/2} \cdot b\]
for certain integers $a$ and $b$. From the last equality, which holds by inductive hypothesis, we have that $d^{[n/2]} \mid N_n(x)$. Similarly, we can get the same result for $D_n(x)$.
\end{proof} 

For proving that all the integer polynomial solutions of \eqref{eq:pell-gen} are given in the previous theorem, we need the following lemma.

\begin{lemma} \label{lemma:l}
Let $f(x)$, $P(x)$ and $Q(x)$ be polynomials of degree $m$, $n m$ and $(n-1) m$, respectively, with coefficients of the highest degree that are positive and 
\begin{equation} \label{eq:pell-d} P^2 - (f^2(x) + d)Q^2 =(-d)^n, \end{equation}
with $d \in \mathbb Z^*$. Given $P'(x) = -\cfrac{f(x)}{d} P(x) + \cfrac{f^2(x)+d}{d}Q(x)$ and $Q'(x) = \cfrac{1}{d} P(x) - \cfrac{f(x)}{d} Q(x)$ then
\begin{enumerate}
\item $P'^2 -(f^2(x)+d)Q'^2 = (-d)^{n-1}$
\item $\deg P' < \deg P$ and $\deg Q' < \deg Q$
\end{enumerate} 
\end{lemma}
\begin{proof}
\begin{enumerate}
\item We have
\[P'^2 -(f^2(x)+d)Q'^2 = \cfrac{f^2(x)}{d^2}P^2 + \cfrac{(f^2(x)+d)^2}{d^2}Q^2 - (f^2(x)+d)\cfrac{1}{d^2}P^2 - (f^2(x)+d) \cfrac{f^2(x)}{d^2}Q^2 = \]
\[= -(f^2(x)+d) \cfrac{1}{d^2}(P^2 - (f^2(x)+d)Q^2) + \cfrac{f^2(x)}{d^2} (P^2 - (f^2(x)+d)Q^2) =\]
\[= -(f^2(x)+d)(-d)^{n-2}+f^2(x)(-d)^{n-2} = (-d)^{n-1}.\]
\item 
Since $\deg P^2 = 2 n m$, $\deg f^2 Q^2 = 2 n m$, $\deg Q^2 = 2 n m - 2 m$ and it holds
\[P^2 - f^2(x) Q^2 -d Q^2 = (-d)^n \]
we have that $P^2 - f^2 Q^2$ has degree $2 n m - 2 m$. Moreover, we can observe that the degree of $P + f Q$ is $n m$ by the hypothesis that coefficients of the highest degree of $P, f, Q$ are positive. Thus, from $P^2 - f^2 Q^2 = (P - f Q) (P + f Q)$, we have that the coefficients of degree $nm, nm-1, \ldots nm-m$ of the polynomial $P - f Q$ are zero, i.e., $\deg Q' < \deg Q$. Similar considerations prove that $\deg P' < \deg P$.
\end{enumerate}
\end{proof}

\begin{theorem}
All the integer polynomial solutions of
\begin{equation} \label{eq:pell-prop-sol} P^2 - (f^2(x)+d)Q^2 = 1 \end{equation}
are the polynomials $N_n(x)$ and $D_n(x)$ for $d = -1$ and any $n$, $\frac{N_{n}(x)}{(-d)^{n/2}}$ and $\frac{D_{n}(x)}{(-d)^{n/2}}$ for $d = 1, \pm 2$ and $n$ even.
\end{theorem}
\begin{proof}
Let $P(x)$ and $Q(x)$ be solutions of \eqref{eq:pell-prop-sol}, if $\deg f = m$, then must be $\deg P = mn$ and $\deg Q = m(n - 1)$. Now we consider the polynomials $\tilde P(x) = d^{n/2} \cdot P(x)$ and $\tilde Q(x) = d^{n/2} \cdot Q(x)$ and we show that these polynomials coincide with the Rédei one's (unless the sign). We proceed by induction on $n$. The basis of the induction is straightforward to check. 
By Lemma \ref{lemma:l} and the inductive hypothesis $-\cfrac{f(x)}{d} \tilde P(x) + \cfrac{f^2(x)+d}{d}\tilde Q(x) = N_{n-1}(x)$ and $\cfrac{1}{d} \tilde P(x) - \cfrac{f(x)}{d} \tilde Q(x) = D_{n-1}(x)$. Let us observe that we have assumed the positivity of the highest coefficients because of the form of the equation \eqref{eq:pell-prop-sol}, in which only the squares $f^2, P^2, Q^2$ appear, hence hypotheses of Lemma \ref{lemma:l} are satisfied.
Thus, from \eqref{eq:matrices}, we have the thesis.
\end{proof}

We have seen that the Rédei polynomials allow to study and solve a vast class of polynomial Pell equations in a very simple and direct way. This also generalizes the result of Nathanson \cite{Nat} that we can retrieve when $f(x) = x$. Moreover, we can use the above approach for studying the most general case of polynomial Pell equations, i.e., 
\[P^2 - f(x) Q^2 = 1,\]
where $f$ is any integer polynomial. In this case, if we consider $\alpha = f(x)$ and $z = \pm \sqrt{f(x) + 1}$, surely we obtain that
\begin{equation*} \label{eq:pell-gen2} N_n^2(x) - f(x) D_n^2(x) = 1\end{equation*}
for any $n \geq 0$, where $N_n(x) = N_n(f(x), \pm \sqrt{f(x) + 1})$ and $D_n(x) = D_n(f(x), \pm \sqrt{f(x) + 1})$. If $f(x)$ is a polynomial such that $f(x) + 1$ is a square, then the Rédei polynomials $N_n(x)$ and $D_n(x)$ are integer polynomials for any $n \geq 0$. 

\begin{question}
Are all the solutions of the polynomial Pell equation \eqref{eq:pell}, for any $D(x) \in \mathbb Z[X]$, the Rédei polynomials?
\end{question}

\begin{example}
In the following Table \ref{table:1}, we write the Rédei polynomials $N_n(x^4-1, x^2)$ and $D_n(x^4-1, x^2)$,  by Theorem \ref{thm:main} we have that they are the solutions of the polynomial Pell equation $P^2 - (x^4 - 1) Q^2 = 1$.
\begin{table}[hp]
\caption{Polynomial solutions of $P^2 - (x^4 - 1) Q^2 = 1$}
\centering
{
\begin{tabular}{|c|c|c|}
\hline 
 $n$ & $N_n(x^4-1, x^2)$ & $D_n(x^4-1, x^2)$  \cr \hline \hline
 1 & $x^2$ & 1  \cr \hline
 2 & $2 x^4 - 1$ & $2 x^2$  \cr \hline
 3 & $4 x^6 - 3 x^2$ & $4 x^4 - 1$  \cr \hline
 4 & $8 x^8 - 8 x^4 + 1$ & $8 x^6 - 4 x^2$  \cr \hline
 5 & $16 x^{10} - 20 x^6 + 5 x^2$ & $16 x^8 - 12 x^4 + 1$ \cr \hline
\end{tabular}
}
\label{table:1}
\end{table}
\end{example}

\begin{example}
In Table \ref{table:2}, we summarize the Rèdei polynomials $N_n(x^4+2, x^2)$ and $D_n(x^4+2, x^2)$. 
For obtaining integer polynomials that are the solutions of the polynomial Pell equation $P^2 - (x^4 + 2) Q^2 = 1$, we have to consider the Rédei polynomials with even index $n$ and divide them by $-2^{n/2}$. In this way we the following solutions
\[(-x^4-1, -x^2), \quad (2x^8 + 4x^4 + 1, 2x^6 + 2x^2), \quad (-4x^{12}-12x^8-9x^4-1, -4 x^{10} -8x^6 - 3x^2), \ldots\]
\begin{table}[hp]
\caption{Rèdei polynomials $N_n(x^4+2, x^2)$ and $D_n(x^4+2, x^2)$}
\centering
{
\begin{tabular}{|c|c|c|}
\hline 
 $n$ & $N_n(x^4+2, x^2)$ & $D_n(x^4+2, x^2)$  \cr \hline \hline
 1 & $x^2$ & 1  \cr \hline
 2 & $2 x^4 + 2$ & $2 x^2$  \cr \hline
 3 & $4 x^6 + 6 x^2$ & $4 x^4 + 2$  \cr \hline
 4 & $8 x^8 + 16 x^4 + 4$ & $8 x^6 + 8 x^2$  \cr \hline
 5 & $16 x^{10} + 40 x^6 + 20 x^2$ & $16 x^8 + 24 x^4 + 4$ \cr \hline
 6 & $32 x^{12} + 96 x^8 + 72 x^4 + 8$ & $32 x^{10} + 64 x^6 + 24 x^2$ \cr \hline
\end{tabular}
}
\label{table:2}
\end{table}
\end{example}

\begin{example}
If we consider the Rédei polynomials $N_n(x^2+3, x)$ and $D_n(x^2+3, x)$, summarized in Table \ref{table:3}, we have that $N_n(x^2+3, x)/(-d)^{n/2}$ and $D_n(x^2+3, x)/(-d)^{n/2}$ are the solutions of the polynomial Pell equation $P^2 - (x^2 + 3) Q^2 = 1$, however, in this case, we do not have integer polynomials as solutions.

\begin{table}[hp]
\caption{Rèdei polynomials $N_n(x^2+3, x)$ and $D_n(x^2+3, x)$}
\centering
{
\begin{tabular}{|c|c|c|}
\hline 
 $n$ & $N_n(x^2+3, x)$ & $D_n(x^2+3, x)$  \cr \hline \hline
 1 & $x$ & 1  \cr \hline
 2 & $2 x^2 + 3$ & $2 x$  \cr \hline
 3 & $4 x^3 + 9 x$ & $4 x^2 + 3$  \cr \hline
 4 & $8 x^4 + 24 x^2 + 9$ & $8 x^3 + 12 x$  \cr \hline
 5 & $16 x^{5} + 60 x^3 + 45 x$ & $16 x^4 + 36 x^2 + 9$ \cr \hline
\end{tabular}
}
\label{table:3}
\end{table}

\end{example}

In the next scetion, we see that the Rédei polynomials can be generalized in a natural way and they are useful for studying polynomial Pell equations of higher degrees.


\section{Polynomial Pell equation of higher degrees} \label{sec:higher}

The classical Pell equation can be generalized in a natural way to higher degrees. Indeed, we can observe that the Pell equation arises considering the unitary elements of the quotient filed $\mathbb Q[x] / (x^2 - d)$, where $x^2 - d$ is an irreducible polynomial over $\mathbb Q$. Thus, considering the unitary elements of $\mathbb Q[x] / (x^3 - c)$, where $c$ is not a cube, we get the cubic Pell equation
\[x^3 + c y^3 + c z^3 - 3 c x y z = 1,\]
in the unknowns $x, y, z$.
Similarly, we can construct the Pell equations of degree $m$ that is defined by
\begin{equation}  \det \begin{pmatrix} 
x_1 & rx_m & rx_{m-1} & \ldots & rx_2 \cr
x_2 & x_1 & rx_m & \ldots & rx_3 \cr
\vdots & \vdots & \ddots & \ddots & \vdots \cr
x_{m-1} & x_{m-2} & \ldots & x_1 & rx_{m} \cr
x_m & x_{m-1} & x_{m-2} & \ldots & x_1
\end{pmatrix}= 1 \end{equation}
in the unknowns $x_1, \ldots x_m$, where $r$ is not a $m$-th power.
For further details, see \cite{Bar}. Thus, it is natural generalizing the study of the polynomial Pell equation to higher degrees, considering the matrix
\begin{equation} \label{eq:pell-genm} P =  \begin{pmatrix} 
P_1 & RP_m & RP_{m-1} & \ldots & RP_2 \cr
P_2 & P_1 & RP_m & \ldots & RP_3 \cr
\vdots & \vdots & \ddots & \ddots & \vdots \cr
P_{m-1} & P_{m-2} & \ldots & P_1 & RP_{m} \cr
P_m & P_{m-1} & P_{m-2} & \ldots & P_1
\end{pmatrix}, \quad \det P = 1, \end{equation}
where $P_1(x), \ldots, P_m(x)$ are unknown polynomials and $R(x)$ is a given integer polynomial.
Gaunet \cite{Gau} studied the polynomial cubic Pell equation 
\begin{equation} \label{eq:pell-cubic} P_1^3 + R P_2^3 + R^2 P_3^3 - 3R P_1P_2P_3 = 1\end{equation}
when $R(x) = x^3 + a x + b$, characterizing when the equation admits non-trivial solutions and finding them. Here, we study the general polynomial Pell equation of degree $m$, given by equation \eqref{eq:pell-genm}, where $P_1, \ldots P_m$ are unknown polynomial in $\mathbb Z[X]$ and $R(x) = f(x) + r$, with $f(x)$ is any integer polynomial and $r \in \mathbb Z$. As a particular case we obtain solutions of the equation \eqref{eq:pell-cubic} for a different class of polynomials $R$. 

We define the \emph{generalized Rédei polynomials} by means of
\[(z + \sqrt[m]{\alpha})^n = A_n^{(0)}(z, \alpha) + A_n^{(1)}(z, \alpha) \sqrt[m]{\alpha} + \ldots + A_n^{(m-1)}(z, \alpha) \sqrt[m]{\alpha^{m-1}}. \]
In the following, when there is no confusion, we omit the dependence on $z, \alpha$ when we write the generalized Rédei polynomials for the seek of simplicity.
The generalized Rédei polynomials can be obtained by the powers of a particular $m \times m$ matrix. Indeed, by definition, we have
\[A_{n+1}^{(0)} + A_{n+1}^{(1)} \sqrt[m]{\alpha} + \ldots + A_{n+1}^{(m-1)}\sqrt[m]{\alpha^{m-1}} = (A_{n}^{(0)} + A_{n}^{(1)}\sqrt[m]{\alpha} + \ldots + A_{n}^{(m-1)}\sqrt[m]{\alpha^{m-1}})(z + \sqrt[m]{\alpha})\]
i.e.
\[A_{n+1}^{(0)} = z A_{n}^{(0)} + \alpha A_n^{(m-1)}, \quad A_{n+1}^{(i)} = z A_{n}^{i} + A_n^{(i-1)}, \quad i = 1, \ldots, m-1.
\]
Thus, given the matrix
\begin{equation} \label{eq:M} M = \begin{pmatrix} z & 0 & 0 & \ldots &  \alpha \cr 
1 & z & 0 & \ldots & 0 \cr
\vdots & \ddots & \ddots & \vdots & \vdots \cr
0 & \ldots & 1 & z & 0 \cr
0 & \ldots & 0 & 1 & z
 \end{pmatrix}\end{equation}
we can write
\[M \begin{pmatrix} A_n^{(0)} \cr \vdots \cr A_n^{(m-1)} \end{pmatrix} = \begin{pmatrix} A_{n+1}^{(0)} \cr \vdots \cr A_{n+1}^{(m-1)} \end{pmatrix}\]
from which
\begin{equation} \label{eq:mat-red-gen} M^n = \begin{pmatrix} A_n^{(0)} & \alpha A_n^{(m-1)} & \ldots & \alpha A_n^{(1)} \cr A_n^{(1)} & A_n^{(0)} & \ldots & \alpha A_n^{(2)} \cr \vdots & \vdots & \ddots & \vdots \cr A_n^{(m-1)} & A_n^{(m-2)} & \ldots & A_n^{(0)} \end{pmatrix}. \end{equation}
From \eqref{eq:M}, we can observe that $\det M^n = (z^m + (-1)^{m-1} \alpha)^n$, on the other hand, from \eqref{eq:mat-red-gen} we have that $\det M^n = \det P$ when $P_1 = A_n^{(0)}, \ldots, P_m = A_n^{(m-1)}$ and $R = \alpha$. Thus, the generalized Rédei polynomials can be exploited for solving the polynomial Pell equation of higher degrees for convenient choices of $z$ and $\alpha$. Indeed, if $z = f(x) \in \mathbb Z[x]$ and $\alpha = (-f(x))^m + r$, with $r \in \mathbb Z$, then the Rédei polynomials satisfy the polynomial equation $\det P = ((-1)^{m-1} r)^n$, where $R(x) = (-f(x))^m + r$, and the polynomials $\cfrac{A_n^{(0)}(x)}{((-1)^{m-1}r)^{n/m}}, \ldots, \cfrac{A_n^{(m-1)}(x)}{((-1)^{m-1}r)^{n/m}}$ satisfy the polynomial Pell equation $\det X = 1$, for $R(x) = (-f(x))^m + r$. However, these solutions are not ever integer polynomials.
They are integer polynomials in the following cases
\begin{enumerate}
\item $r = -1$ and any $n \geq 0$,
\item $r = 1$ and any $n \equiv 0 \pmod m$,
\item $r = \pm m$ and any $n \equiv 0 \pmod m$, when $m$ is a prime number.
\end{enumerate}
The situations 1 and 2 are immediate to verify. Let us focus on situation 3. We can observe that
\[A_1^{(0)}(x) = f(x), A_2^{(0)}(x) = f^2(x), \ldots, A_{m-1}^{(0)}(x) = f^{m-1}(x), A_m^{(0)}(x) = \pm m\]
Moreover, the characteristic polynomial of $M$ is
\[ x^m + \sum_{i=1}^{m-1} (-1)^i \binom{m}{i} x^{m-i} f^i(x) \pm m,  \]
thus if $m$ is prime, we have
\[\pm m | A_m^{(0)}(x), \pm m | A_{m+1}^{(0)}(x), \ldots, \pm m|A_{2m-1}^{(0)}(x)\]
and consequently $m^2|A_{2m}^{(0)}(x)$. Thus we can prove by induction that the Rédei polynomials are in $\mathbb Z[x]$ for $r = \pm m$ and any $n \equiv 0 \pmod m$, when $m$ is a prime number. 

\begin{question}
Are all the integer polynomial solutions of \eqref{eq:pell-genm}, for $R(x) = (-f(x))^m + r$, the Rédei polynomials?
\end{question}

\section*{Acknowledgments}
The author is really grateful to the anonymous referee for the carefully reading of the paper and for all the suggestions that improved the presentation of the paper.

\end{document}